\documentclass[11pt]{amsart}
\usepackage{graphicx, xcolor}
\usepackage{amssymb}
\numberwithin{equation}{section}

\def\N{\mathbb{N}}
\def\R{\mathbb{R}}

\renewcommand\d{\partial}
\renewcommand\a{\alpha}
\renewcommand\b{\beta}

\newcommand\s{\sigma}

\newcommand{\newatop}{\genfrac{}{}{0pt}{1}}
\def\de{\delta}

\def\l{\lambda}

\def\epsilon{\varepsilon}
\def\e{\varepsilon}

\newcommand\br{\begin{rem}}
\newcommand\er{\end{rem}}
\newcommand\bp{\begin{pmatrix}}
\newcommand\ep{\end{pmatrix}}
\newcommand\be{\begin{equation}}
\newcommand\ee{\end{equation}}
\newcommand\ba{\begin{equation}\begin{aligned}}
\newcommand\ea{\end{aligned}\end{equation}}

\newtheorem{theorem}{Theorem}[section]

\newtheorem{lemma}[theorem]{Lemma}

\newtheorem{remark}[theorem]{Remark}
\newtheorem{ans}[theorem]{Definition}

\title{Dirichlet boundary conditions for degenerate and singular nonlinear parabolic equations} 

\begin{document}

\maketitle
\begin{center}
\textsc{\textmd{Fabio Punzo\footnote{Dipartimento di Matematica
``F. Enriques", Universit\`a degli Studi di Milano, via C. Saldini
50, 20133 Milano, Italy. Email: fabio.punzo@unimi.it.}, Marta
Strani \footnote{Dipartimento di Matematica e Applicazioni,
Universit\`a di Milano Bicocca, via Cozzi 55, 20125 Milano, Italy.
Email: marta.strani@unimib.it, martastrani@gmail.com .}}}
\end{center}

\vskip1cm

\begin{abstract}
We study existence and uniqueness of solutions to a class of
nonlinear degenerate parabolic equations, in bounded domains. We
show that there exists a unique solution which satisfies possibly
inhomogeneous Dirichlet boundary conditions. To this purpose some
barrier functions are properly introduced and used.
\end{abstract}

\pagestyle{myheadings}
\thispagestyle{plain}
\markboth{F. PUNZO, M.STRANI}{DIRICHLET BOUNDARY CONDITIONS FOR NONLINEAR PARABOLIC EQUATIONS}


\section{Introduction}
We are concerned with bounded solutions to the following nonlinear
parabolic equation:
\begin{equation}\label{intro0}
\rho \, \d_t u = \Delta [G(u)]\quad  {\rm in } \  \Omega \times
(0,T],
\end{equation}
where $\Omega $ is an open bounded subset of $\R^N\, (N\geq 1)$ with boundary
$\d\Omega=\mathcal S$ and $\rho$ is a positive function of the
space variables. We always make the following assumption: \vskip0.2cm {\bf H0.} $\mathcal S$ is an
$(N-1)-$dimensional compact submanifold of $\R^N$ of class $C^3$.

\vskip0.2cm \noindent Moreover, we require the functions $\rho$,
$G$ and $f$ to satisfy the following hypotheses

\vskip0.2cm
{\bf H1.}
 $\rho\in C(\Omega), \, \rho>0 \, $ in $\Omega$;
\vskip0.2cm
{\bf H2.}
$ G \in C^1(\R) , \, G(0)=0,\, G'(s)>0\,\hbox{\,for any}\,\,s\in\R\setminus\{0\}.$ Moreover, if $G'(0)=0$, then $G'$ is decreasing  in $ (-\delta,0)$ and increasing in
$(0,\delta) $ for some $\delta >0$.

\smallskip

Clearly, the character of equation \eqref{intro0} is determined by
$G$ and $\rho$; to see this, let us think equation \eqref{intro0}
as
\begin{equation}\label{e611}
\d_t u =\frac1{\rho} \Delta [G(u)]\quad  {\rm in } \  \Omega
\times (0,T]\,, \end{equation} and set
\[d(x):={\rm dist}(x,\mathcal S) \quad (x\in \bar \Omega)\,.\]
In fact, in view of the nonlinear function $G(u)$ and hypothesis
{\bf H2} the equation \eqref{intro0} can be {\it degenerate};
however, we also consider the case that such a kind of degeneracy
does not occur (see {\bf H5} below). Moreover, if the coefficient
$\rho(x)\to 0$ as $d(x)\to 0$, the operator $\frac 1{\rho}\Delta$
has the coefficient $\frac 1{\rho}$ which is unbounded at
$\mathcal S$, so the operator is {\it singular}; whereas, if
$\rho(x)\to \infty$ as $d(x)\to 0$, the operator
$\frac1{\rho}\Delta$ is degenerate at $\mathcal S$.

\vskip0.5cm Problem \eqref{intro0} appears in a wide number of
physical applications (see, e.g., \cite{KR1}); note that, by
choosing $G(u)=|u|^{m-1}u$ for some $m>1$, we obtain the well
known porous medium equation with a variable density
$\rho=\rho(x)$ (see \cite{Eid90, EK}).

\vskip0.2cm In the literature,  a particular attention has been
devoted to the following companion Cauchy problem
\begin{equation}\label{intro}
\left\{\begin{aligned}
\rho \d_t u &= \Delta [G(u)] \quad & {\rm in } \ & \R^N \times (0,T], \\
u&=u_0 \quad &{\rm in}  \ & \R^N \times \{ 0\}.
\end{aligned}\right.
\end{equation}
In particular, existence and uniqueness of   solutions to
\eqref{intro} have been extensively studied; note that here and hereafter we always consider {\it very weak} solutions (see Section \ref{mb} for the precise definition).
To be specific, if
one makes the following assumptions: \vskip0.2cm \noindent $(i)
\;\;   \rho\in C(\R^N), \, \rho>0 \, $,

\vskip0.1cm \noindent $ (ii)\;\; u_0\in L^{\infty}(\R^N)\cap
C(\R^N)$, \vskip0.2cm \noindent it is well known (see
\cite{EK,KR1,GHP, P1}) that there exists a bounded solution to
\eqref{intro}; moreover, for $N=1$ and $N=2$ such a solution is
unique. When $N \geq 3$, the uniqueness of the solution in the
class of bounded functions is no longer guaranteed, and it is
strictly related with the behavior at infinity of the density
$\rho$. Indeed, it is possible to prove that if $\rho$ does not
decay too fast at infinity, then problem \eqref{intro} admits at
most one bounded solution (see \cite{P1}). On the contrary, if one
suppose that  $\rho$ decays sufficiently fast at infinity, then
the non uniqueness appears (see \cite{Eid90, GMP, KPT,P1}).

\vskip0.2cm

In this direction, in  \cite{GMP} the authors prove the existence
and uniqueness of the solution to \eqref{intro} which satisfies
the following additional condition at infinity
\begin{equation}\label{e6}
\lim_{|x|\to \infty} u(x,t)=a(t)\quad \textrm{uniformly for}\,\,\,
t\in [0,T]\,,
\end{equation}
supposing $a \in C([0,T]), a>0$ and
$\lim_{|x|\to\infty}u_0(x)=a(0).$ Note that \eqref{e6} is a
point-wise condition at infinity for the solution $u$. Also, the
results of \cite{GMP} have been generalized in \cite{KP, KP2} to the
case of more general operators.

\vskip0.2cm

When considering equation \eqref{intro0} in a bounded subset
$\Omega \subset \R^N$, in view of {\bf H1}, since $\rho$ is
allowed either to vanish or to diverge at $\mathcal S$, it is
natural to consider the following initial value problem associated
with \eqref{intro0}:
\begin{equation}\label{intro2}
\left\{\begin{aligned}
\rho \d_t u &= \Delta [G(u)] \quad & {\rm in } \ & \Omega \times (0,T], \\
u&=u_0 \quad &{\rm in}  \ & \Omega \times \{ 0\},
\end{aligned}\right.
\end{equation}
where no boundary conditions are specified at $\mathcal S$. We
require $\rho$, $G$  and $f$ to satisfy hypotheses {\bf H1-2-3};
furthermore, for the initial datum $u_0$ we assume that
\vskip0.2cm {\bf H3.} $ u_0\in L^{\infty}(\Omega)\cap C(\Omega).$
\vskip0.2cm Concerning the existence and uniqueness of the
solutions to \eqref{intro2}, the case $G(u)=u$ has been largely
investigated, using both analytical and stochastic methods (see,
e.g., \cite{Khas, PoPT0, PoPT, Tai}). Also analogous elliptic or
elliptic-parabolic equations have attracted much attention in the
literature (see, e.g., \cite{FP1, F, FP2, Fich, Fich2, Freid, MP, 
OR}); in particular,  the question of prescribing continuous
data at $\mathcal S$ has been addressed (see, e.g., \cite{Khas,
OR, PoPT0, PoPT}).

For general nonlinear function $G$, the well-posedness of problem \eqref{intro2} has been studied in
\cite{KT} in the case $N=1$ and subsequently addressed for $N\geq
1$ in \cite{P7}.

Precisely, in \cite{P7} is proven that, if $\rho$ diverges
sufficiently fast as $d(x) \to 0$, then one has uniqueness of
bounded solutions not satisfying any additional condition at
$\mathcal S$.

Indeed, if one requires that there exists $\hat \e>0$ and
$\underline \rho \in C((0,\hat \e])$ such that \vskip0.1cm
\noindent $\bullet \;\; \rho(x)\geq \underline \rho(d(x)) >0$, for
any $x \in \mathcal S^{\hat \e}:= \{ x \in \Omega \, | \,d(x) <
\hat \e \}$, \vskip0.1cm \noindent $\bullet \;\; \int_0^{\hat \e}
\eta \, \underline\rho(\eta) \, d \eta=+\infty$, \vskip0.2cm
\noindent then there exists at most one bounded solution to
\eqref{intro2}.

\vskip0.2cm \noindent Conversely, if either $\rho(x) \to \infty$
sufficiently slow or $\rho$ does not diverge  when $d(x)\to 0$,
then nonuniqueness prevails in the class of bounded solutions.
Precisely, in \cite{P7} it is proven that, if there exists $\hat
\e>0$ and $\overline \rho \in C((0,\hat \e])$ such that
\vskip0.1cm \noindent $\bullet \;\; \rho(x)\leq \overline
\rho(d(x))$, for any $x \in \mathcal S^{\hat \e}$, \vskip0.1cm
\noindent $ \bullet\;\; \int_0^{\hat \e} \eta \, \overline\rho(\eta) \, d
\eta<+\infty$, \vskip0.2cm \noindent then, for any $A \in {\rm
Lip}([0,T])$, $A(0)=0$,  there exists a solution to \eqref{intro2}
satisfying
 \begin{equation}\label{lavoroPunzo}
 \lim_{d(x) \to 0} |U(x,t)-A(t)|=0,
 \end{equation}
uniformly with respect to $t \in [0,T]$, where  $U$ is defined as
\begin{equation*}
U(x,t) := \int _0^t G(u(x,\tau)) \, d\tau.
\end{equation*}
In particular, the previous result implies non-uniqueness of
bounded solutions to \eqref{intro2}. Moreover, the solution to
problem \eqref{intro2} which satisfies \eqref{lavoroPunzo} is
unique, provided $A\equiv 0$ or $G(u)=u.$

\vskip0.2cm Formally, the boundary $\mathcal S$ for problem
\eqref{intro2} plays the same role played by {\it infinity} for
the Cauchy problem \eqref{intro}; hence, the well-posedness for
\eqref{intro2} depends on the behavior of $\rho$ in the limit
$d(x)\to 0$, in analogy with the previous results for the Cauchy
problem \eqref{intro}, where it depends on the behavior of $\rho$
for large $|x|$.

Thus, a natural question that arises is if it is possible to
impose at $\mathcal S$ Dirichlet boundary conditions, instead of
the integral one \eqref{lavoroPunzo}. Moreover, on can ask if such
a Dirichlet condition restore uniqueness in more general
situations than the ones considered in connection with
\eqref{lavoroPunzo}. Observe that, as recalled above, the same
question has already been investigated for the linear case
$G(u)=u$ (see, e.g., \cite{Khas, OR, PoPT0, PoPT}), and for the
case that $\rho\equiv 1$ and $G$ is general (see \cite{DiB1,
DiB3}). The case where both $\rho$ and $G$ are general, which is a
quite natural situation also for various applications (see, e.g.,
\cite{KR}), has not been treated in the literature and is the
object of our investigation.

In fact, the main novelty of our paper relies in the following
result: we prove existence and uniqueness of a bounded solution to
problem \eqref{intro2} satisfying Dirichlet possibly
non-homogeneous boundary conditions. This is of course a much
stronger condition with respect to \eqref{lavoroPunzo}. As in
\cite{P7}, we require the function $\rho$ to satisfy \vskip0.2cm
{\bf H4.} there exists $\hat \e>0$ and $\overline \rho \in
C((0,\hat \e])$ such that \vskip0.1cm ${\bf i.} \;\; \rho(x)\leq
\overline \rho(d(x)) $, for any $x \in \mathcal S^{\hat \e}$,
\vskip0.1cm $ {\bf ii.}\;\; \int_0^{\hat \e} \eta \,
\overline\rho(\eta) \, d \eta<+\infty$. \vskip0.2cm \noindent A
natural choice for $\overline \rho$ is given by
\begin{equation}\label{choicerho}
\overline \rho(\eta)=\eta^{-\alpha}, \ \ {\rm for \ some} \ \alpha \in (-\infty, 2), \ {\rm and} \ \eta  \in (0,\hat \e].
\end{equation}

\noindent Under the hypothesis {\bf H4}, we show that, for any
$\varphi\in C(\mathcal S\times [0,T])$ , if either $G$ is non
degenerate, i.e. there holds \vskip0.2cm {\bf H5.} $G\in
C^1(\R),\; G'(s)\ge\a_0>0  \quad  \textrm{for any}\;\; s\in\R\,,$
\vskip0.2cm
 \noindent or $\varphi$ and $u_0$ satisfy
\begin{equation}\label{ipopositiva}
\varphi >0 \quad \textrm{in}\;\; \mathcal S\times[0, T]\,,\;\;
\liminf_{x\to x_0} u_0(x)\geq \a_1>0\quad \textrm{for every}\;\,
x_0\in \mathcal S\,,
\end{equation}
then there exists a unique bounded solution to \eqref{intro2} such
that, for each $\tau \in (0,T)$,
\begin{equation}\label{secres}
\lim_{\newatop{x\to x_0}{t\to t_0}} u(x,t)=\varphi(x_0,t_0) \quad
\textrm{uniformly with respect to} \,\,\, t_0 \in [\tau,T] \, \,
{\rm and } \,\, x_0 \in \mathcal S.
\end{equation}
If we drop either the assumption of non-degeneracy on $G$ or the
assumption \eqref{ipopositiva}, we need to restrict our analysis
to the special class of  data $\varphi$ which only depend on $x$;
in fact, for any $\varphi \in C(\mathcal S)$ we prove that there
exists a unique bounded solution to \eqref{intro2} satisfying, for
each $\tau \in (0,T)$,
\begin{equation}\label{secresa}
\lim_{x \to x_0} u (x,t) =\varphi(x_0) \quad \textrm{uniformly
with respect to} \,\,\, t \in [\tau,T] \, \, {\rm and } \,\, x_0
\in \mathcal S,
\end{equation}
provided
\begin{equation}\label{e301}
\lim_{x\to x_0} u_0(x)=\varphi(x_0)\quad \textrm{for every}\,\,
x_0\in \mathcal S\,.
\end{equation}

To prove the existence results we introduce and use suitable
barrier functions (see \eqref{e600}, \eqref{e601}, \eqref{e602},
\eqref{e603}, \eqref{ea3}, \eqref{e604}, \eqref{e606},
\eqref{e607} below). We should note that the definitions of such
barriers seem to be new. Let us observe that in constructing such barrier functions, always supposing that {\bf H4} holds, the cases
$\inf_{\Omega} \rho
>0$ and $\rho \in L^\infty(\Omega)$ will be treated separately
(for more details, see Section 3). To explain the differences among these two cases, let
us refer to the model case, in which hypothesis {\bf H4} holds
with $\overline \rho$ given by \eqref{choicerho}. So, the previous
two cases correspond to the choices $\alpha <0$ and $\alpha \in
[0,2)$, respectively. In view of \eqref{e611}, it appears natural
that the operator $\frac 1{\rho}\Delta$ has a prominent role. From
this viewpoint we can say that the previous two cases are deeply
different, since, when $\alpha\in (0,2)$, the operator $\frac
1{\rho}\Delta$ is degenerate at $\mathcal S$, whereas, when
$\a<0$, it is singular, in the sense that its coefficient $\frac
1{\rho}$ blows-up at $\mathcal S$. Clearly, the choice $\a=0$
recasts in both cases.

In constructing our barrier functions, besides taking into account
the behavior at $\mathcal S$ of the density $\rho(x)$ as described above, we have
to overcome some difficulties due to the nonlinear function
$G(u)$. In this respect, we should note that on the one hand,
barrier functions similar to those we construct were used in
\cite{GMP} and in \cite{KP}, where problem \eqref{intro} was addressed and conditions were prescribed at infinity.
However, such barriers cannot be trivially adapted to our case.
Indeed, by an easy variation of them we could only consider
$\mathcal S$ in place of {\it infinity}, prescribing $u(x,t)\to
a(t)$ as $d(x)\to 0\,\, (t\in (0,T])$, but we cannot distinguish
different points $x_0\in \mathcal S$ and impose conditions
\eqref{secres} and \eqref{secresa}\,. On the other hand, other
similar barriers were used in the literature (see, e.g.,
\cite{Fried}) to prescribe Dirichlet boundary conditions to
solutions to {\it linear} parabolic equations, in bounded domains; similar results have also been established for linear elliptic equations (see \cite{GT}, \cite{Mi});
however, they cannot be used in our situation, in view of the
presence of the nonlinear function $G(u)$.

\vskip0,2cm

Let us mention that our results have some connections with
regularity results up to the boundary. In fact, as a consequence
of our existence and uniqueness results, any solution to problem
\eqref{intro2} is continuous in $\overline\Omega\times[0,T].$
Similar regularity results could be deduced from results in
\cite{DiB1} and in \cite{DiB3}, where more general equations are
treated, only when
\begin{equation}\label{e100}
C_1\leq \rho(x)\leq C_2\quad \textrm{for all}\;\; x\in \Omega\,,
\end{equation}
for some $0<C_1<C_2$. However, we suppose hypotheses  {\bf H1} and
{\bf H5}, that are weaker than \eqref{e100}\,.

\medskip

We close this introduction with a brief overview of the paper. In
Section 2 we present a description of the main contributions of
the paper; in particular, we state Theorem \ref{pointcond},
Theorem \ref{degpuntuale} and Theorem \ref{deg2puntuale}, that
assure, under suitable hypotheses, the existence of a bounded
solution to \eqref{intro2} satisfying a proper Dirichlet boundary
condition. Subsequently, we show that such a solution is unique
(see Theorem \ref{teounicita}).

\noindent Section 3 is devoted to the proofs of the existence results, while in Section 4 the proof of the uniqueness result is given.

\section{Statement of the main results}

In this section we present existence and uniqueness results for
bounded solutions to
\begin{equation}\label{mainproblem}
\left\{\begin{aligned}
\rho \d_t u &= \Delta [G(u)] \quad & {\rm in } \ & \Omega \times (0,T], \\
u&=u_0 \quad &{\rm in}  \ & \Omega \times \{ 0\},
\end{aligned}\right.
\end{equation}
where $\Omega \subset \R^N$ satisfies hypothesis {\bf H0}, and  $\rho$, $G$ and $u_0$ satisfy hypotheses {\bf H1-4}.

In the following, we will extensively use the following notations:
\begin{itemize}
\item $Q_T := \Omega \times (0,T]$; \vskip0.2cm \item $\mathcal
S^\e:= \{ x \in \Omega \, : d(x) < \e \}\;\; (\e>0)$; \vskip0.2cm
\item $\mathcal A^\e := \d \mathcal S^\e \cap \Omega$; \vskip0.2cm
\item $\Omega^\e := \Omega \setminus \mathcal S^\e$\,.
\end{itemize}

\subsection{Mathematical background}\label{mb} Before stating our results, let us define the tools we shall use in the following.

\begin{ans}\label{soluzdebole}
A function $u \in C(\Omega \times [0,T ]) \cap L^{\infty}(\Omega
\times (0,T) )$ is a solution to \eqref{mainproblem} if
\begin{equation}\label{e8a}
\begin{aligned}
\int_0^{\tau}\int_{\Omega_1}\big[ u \, \rho\, \d_t \psi  + &G(u) \Delta \psi\big]\, dx \, dt = \\
&= \int_{\Omega_1} \,\big[u(x,T)\psi(x,T)  - u_0(x)
\psi(x,0)\big]\rho(x)\,dx \\ &+ \int_0^{\tau}\int_{\d \Omega_1}
G(u) \langle \nabla \psi, \nu \rangle  \, dS \, dt,
\end{aligned}
\end{equation}
for any open set $\Omega_1$ with smooth boundary $\d \Omega_1$
such that $\overline \Omega_1 \subset \Omega$, for any $\tau \in
(0,T]$ and for any $\psi \in C^{2,1}_{x,t}( \overline \Omega_1
\times [0,\tau])$, $\psi \geq 0$, $\psi =0$ in $\d \Omega_1 \times
[0,\tau]$, where $\nu$ denotes the  outer normal to $\Omega_1$.

Moreover, we say that $u$ is a supersolution (subsolution respectively) to \eqref{mainproblem} if \eqref{e8a} holds with $\leq$ ( $\geq respectively$).

\end{ans}

Given $\e>0$, we also consider the following auxiliary problem
\begin{equation}\label{mainproblemeps}
\left\{\begin{aligned}
\rho \d_t u &= \Delta [G(u)] \quad & {\rm in } \ & \Omega^\e \times (0,T]:=Q^\e_T, \\
u&=\phi \quad &{\rm in} \ & \mathcal A^\e \times (0,T), \\
u&=u_0 \quad &{\rm in}  \ & \Omega^\e \times \{ 0\};
\end{aligned}\right.
\end{equation}
where $\phi\in C(\mathcal A^\e\times [0,T])\,, \phi(x,0)=u_0(x)$
for all $x\in \mathcal A^\e$\,.

\begin{ans}\label{soluzdeboleeps}
A function $u \in C(\overline{\Omega^\e} \times [0,T ])$ is a
solution to \eqref{mainproblem} if
\begin{equation}\label{e8aeps}
\begin{split}
 \int_0^{\tau}\int_{\Omega_1}\big[ u \, \rho\, \d_t \psi  + G(u) \Delta \psi\big]\, dx \, dt
= \int_{\Omega_1} \,\big[u(x,T)\psi(x,T)  - u_0(x)
\psi(x,0)\big]\rho(x)\,dx \\ + \int_0^{\tau}\int_{\d
\Omega_1\setminus \mathcal A^\e} G(u) \langle \nabla \psi, \nu
\rangle  \, dS \, dt + \int_0^{\tau}\int_{\d \Omega_1\cap \mathcal
A^\e} G(\phi) \langle \nabla \psi, \nu \rangle  \, dS \, dt,
\end{split}
\end{equation}
for any open set $\Omega_1\subset \Omega^\e$ with smooth boundary
$\d \Omega_1$, for any $\tau \in (0,T]$ and for any $\psi \in
C^{2,1}_{x,t}( \overline \Omega_1 \times [0,\tau])$, $\psi \geq
0$, $\psi =0$ in $\d \Omega_1 \times [0,\tau]$, where $\nu$
denotes the  outer normal to $\Omega_1$. Supersolution  and
subsolution  are defined accordingly.

\end{ans}

\subsection{Existence results}
At first, we consider the case of nondegenerate nonlinarities $G$
satisfying hypothesis {\bf H5}.

.

\begin{theorem}\label{pointcond}
Let hypotheses {\bf H0-H1}, {\bf H3-H5} be satisfied. Let $\varphi
\in C(\mathcal S \times [0,T])$. Then there exists a maximal
solution to \eqref{mainproblem} such that, for each $\tau \in
(0,T)$,
\begin{equation}\label{primoris}
\lim_{\newatop{x\to x_0}{t\to t_0}} u(x,t) = \varphi(x_0,t_0),
\end{equation}
uniformly with respect to $t_0 \in [\tau,T]$ and $x_0 \in \mathcal
S$.

\end{theorem}

We can also prove similar results to Theorem \ref{pointcond} in
the case of a general nonlinearity $G$ satisfying {\bf H2}\,.

\begin{theorem}\label{degpuntuale}
Let hypotheses {\bf H0-4} be satisfied and let $\varphi \in
C(\mathcal S)$. Suppose that condition \eqref{e301} holds. Then
there exists a maximal solution to \eqref{mainproblem} such that
\begin{equation}\label{secris}
\lim_{x\to x_0} u(x,t) = \varphi(x_0),
\end{equation}
uniformly with respect to $t \in [0,T]$ and $x_0\in \mathcal S$.
\end{theorem}

Finally, we can also consider data $\varphi$ and $u_0$ satisfying
\begin{equation}\label{ipophipositiva}
\varphi>0 \quad \textrm{in}\;\,\mathcal S\times [0,T] \quad
\textrm{and}\;\; \liminf_{x\to x_0}u_0(x)\geq \alpha_1>0\quad
\textrm{for every}\;\, x_0\in \mathcal S\,.
\end{equation}

\begin{theorem}\label{deg2puntuale}
Let hypothesis {\bf H0-4} be satisfied and let $\varphi\in
C(\mathcal S\times [0,T])$. Suppose that \eqref{ipophipositiva}
holds. Then there exists a maximal solution to \eqref{mainproblem}
such that \eqref{primoris} holds.

\end{theorem}

\begin{remark}\label{osst0}
{\rm If we further suppose that
\begin{equation}\label{e302}
\lim_{x\to x_0} u_0(x)=\varphi(x_0,0)\quad \textrm{for every}\,\,
x_0\in \mathcal S\,,
\end{equation}
then in Theorems \ref{pointcond} and \ref{deg2puntuale} we can
take $\tau=0$.}
\end{remark}

\subsection{Uniqueness results}

\begin{theorem}\label{teounicita}
Let hypotheses {\bf H0-4} be satisfied, and  let $\varphi\in
C(\mathcal S\times [0,T])$. Suppose that \eqref{ipophipositiva}
holds. Then there exists at most one bounded solution to
\eqref{mainproblem} such that \eqref{primoris} holds.
\end{theorem}

\begin{remark}\label{osst1}
{\rm If we consider either the case of a non-degenerate
nonlinearity $G$ satisfying {\bf H5}, or if we require that
$\varphi(x,t)=\varphi(x)$ for all $t\in (0,T]$, the previous
uniqueness result still holds. It can be shown by using the same
arguments as in Theorem \ref{teounicita}. }
\end{remark}

\section{Existence results: proofs}

\subsection{Preliminaries}

In the proofs of our existence results, in order to show that the solution we construct is {\it maximal}, we will make use of the following
lemma.
\begin{lemma}\label{lemma-1}
Let hypotheses {\bf H0-4} be satisfied. Let  $u$ be a subsolution
to problem \eqref{mainproblem} and let $ \hat u$ be a
supersolution to problem \eqref{mainproblem}. Suppose that for
each $\tau\in (0,T)$ there exists $\e_\tau>0$ such that, for all
$0<\e<\e_\tau$,
\begin{equation}\label{e605}
u \leq \hat u \quad {\rm in } \quad \mathcal A^\e \times (\tau,T].
\end{equation}
Then
\begin{equation*}
u \leq \hat u \quad {\rm in } \quad Q_T.
\end{equation*}
\end{lemma}

\noindent In order to prove Lemma \ref{lemma-1}, we need to state
the following result; for its proof, see \cite[Lemma 10]{ACP}.

\begin{lemma}\label{lemma-1bis}
Let $\e>0.$ Let
\begin{equation}\label{defa}
a:=\left\{\begin{aligned}
&[G(u)-G(\hat u)]/(u-\hat u) \quad &{\rm for} \ \ u \neq \hat u, \\
&0 \quad & {\rm elsewhere},
\end{aligned}\right.
\end{equation}
with $u$ and $\hat u$ as in Lemma \ref{lemma-1}. Then there exists
a sequence $\{a_n\}\in C^\infty(\overline{Q^\e_T})$ such that
\begin{equation*}
\frac{1}{n^{N+1}} \leq a_n \leq \| a\|_{L^\infty(Q^\e_T)} +
\frac{1}{n^{N+1}} \qquad {\rm and} \qquad
\frac{(a_n-a)}{\sqrt{a_n}} \to 0 \ \ {\rm in} \ \ L^2(Q^\e_T).
\end{equation*}
Furthermore, let $\chi \in C^\infty_0(\Omega^\e)$ with $0 \leq
\chi \leq 1$. Then there exists a unique solution $\psi_n \in
C^{2,1}_{x,t}(\overline {Q^\e_T})$ to problem
\begin{equation}\label{BPP}
\left \{ \begin{aligned}
\rho \d_t \psi_n +  a_n\Delta \psi_n&=0 \quad & {\rm in } \ &Q^\e_T, \\
\varphi_n(x,T)&=\chi(x) \quad &{\rm in}  \ & \Omega^\e.
\end{aligned}\right.
\end{equation}
Moreover, $\psi_n$ has the following properties: \vskip0.2cm
\noindent  {\bf i.} $0 \leq \psi_n \leq 1$ on $\overline
Q^\e_T\,$; \vskip0.2cm \noindent {\bf ii.} $\int \int_{Q^\e_T} a_n
|\Delta \psi_n|^2 <C\, $, for some $C >0$ independent of $n$.
\vskip0.2cm \noindent  {\bf iii.} $\sup_{0 \leq t \leq T}
\int_{\Omega^\e} |\nabla \psi_n|^2 < C\, $, for some $C >0$
independent of $n$. \vskip0.2cm \noindent
\end{lemma}

\begin{proof}[\bf Proof of Lemma \ref{lemma-1}]
The proof of this lemma is an adaptation of the arguments used in
\cite[Proposition 9]{ACP}. Let $a$ be as in \eqref{defa}; since
$u$ and $\hat u$ are respectively subsolution and supersolution to
\eqref{mainproblem}, in view of the Definition \ref{soluzdebole},
with $\Omega_1$ and $\psi$ as in Definition \ref{soluzdeboleeps},
by \eqref{e8aeps} with $\tau=T$, we get
\begin{equation}\label{soprasotto}
\begin{split}
\int_{\Omega^\e} \rho(x) [u(x,T)-\hat u(x,T)] \psi(x, T) \, dx -\int_0^T \int_{\Omega^\e} (u-\hat u) \left \{\d_t \psi+a \, \Delta \psi \right\} dt \, dx \leq \\
 - \int_0^{\tau} \int_{\mathcal A^\e}[G(u)-G(\hat u)]  \langle \nabla \psi, \nu \rangle dS dt- \int_{\tau}^{T} \int_{\mathcal A^\e}[G(u)-G(\hat u)]  \langle \nabla \psi, \nu \rangle dS dt.
\end{split}
\end{equation}
Now, let $\{a_n\}$ and $\psi_n$ as in Lemma \ref{lemma-1bis}.
Since, for every $n \in \N$, there holds $\langle \nabla \psi_n,
\nu \rangle \leq 0$ on $\mathcal A^\e$, if we set $\psi=\psi_n$ in
\eqref{soprasotto}, using \eqref{e605}, we obtain
\begin{equation}\label{e102}
\begin{split}
\int_{\Omega^\e} \rho [u(x,T)-\hat u(x,T)] \chi(x) \, dx -\int_0^T \int_{\Omega^\e} (u-\hat u)(a-a_n) \, \Delta \psi_n dt \, dx \leq \\
 - \int_0^{\tau} \int_{\mathcal A^\e}[G(u)-G(\hat u)]  \langle \nabla \psi_n, \nu \rangle  dS \, dt-
 \int_{\tau}^{T} \int_{\mathcal A^\e}[G(u)-G(\hat u)]  \langle \nabla \psi_n, \nu \rangle  dS \, dt\leq & \\
  \leq  - \int_0^{\tau} \int_{\mathcal A^\e}[G(u)-G(\hat u)]  \langle \nabla
\psi_n, \nu \rangle \, dS \, dt.
 \end{split}
\end{equation}
In view of Lemma \ref{lemma-1bis}, we get
\begin{equation}\label{e101}
\begin{split}
\left|\int_0^T \int_{\Omega^\e} (u-\hat u)(a-a_n) \, \Delta \psi_n
dt \, dx\right|\leq
C_1\left\|\frac{a-a_n}{\sqrt{a_n}}\right\|_{L^2(Q_T)}\left\|\sqrt{a_n}\Delta\psi_n
\right\|_{L^2(Q_T)}\\ \leq C_1\sqrt C
\left\|\frac{a-a_n}{\sqrt{a_n}}\right\|_{L^2(Q_T)}\to 0 \quad
\textrm{as}\,\, n\to \infty\,,
\end{split}
\end{equation}
where the constant $C_1>0$ depends only on $\|  u\|_{L^\infty}$
and $\| \hat u\|_{L^\infty}$. Furthermore,
\begin{equation}\label{e103}
\begin{split}
\Big | \int_0^{\tau} \int_{\mathcal A^\e}[G(u)-G(\hat u)]   \langle \nabla \psi_n, \nu \rangle dS \, dt \Big | \leq \\
 \leq \left(  \int_0^{\tau} \int_{\Omega^\e}[G(u)-G(\hat
u)]^2 dx \, dt \right)^{\frac{1}{2}}
\left( \int_0^{\tau} \int_{\Omega^\e}\left|\nabla \psi_n\right|^2dx \, dt \right)^{\frac{1}{2}} \leq \\
 \leq C \left( \int_0^{\tau} \int_{\Omega^\e} |\nabla \psi_n|^2 dx \, dt \right)^{\frac{1}{2}}\leq C_1\, \tau\, \sqrt{ C},
\end{split}
\end{equation}
where we used Lemma \ref{lemma-1bis}, ${{(\bf iii)}}$. Hence, in
view of \eqref{e101} and \eqref{e103}, letting $n\to \infty$ in
\eqref{e102} and then $\tau\to 0$,  we end up with
\begin{equation}\label{finest}
\int_{\Omega^\e} \rho(x) [u(x,T)-\hat u(x,T)]  \chi(x) \, dx  \leq
0\,.
\end{equation}
Since \eqref{finest} holds for every $\chi \in
C^{\infty}_0(\Omega^\e)$, by approximation it also holds with
$\chi(x)= {\rm sign}(u(x, T)-\hat u(x, T))^+,\, x\in \Omega^\e$.
This implies $u\leq \hat u$ in $Q^\e_T$, from which the thesis
immediately follows, letting $\e\to 0^+.$
\end{proof}

\subsection{Proofs of the Theorems}

In view of the assumption on $\rho(x)$ given in {\bf H4}, there
holds the following lemma (see \cite{P7}).

\begin{lemma}\label{lemma0}
Let hypotheses {\bf H0- H4} be satisfied. Then there exists a function $V(x) \in C^2(\overline{\mathcal S^\e})$ such that
\begin{equation*}
\left\{\begin{aligned}
\Delta V(x) & \leq -\rho(x), \quad &{\rm for \;\, all}\;\;  x \in \mathcal S^\e, \\
V(x) & > 0, \quad &{\rm for \;\, all}\;\; x \in \mathcal S^\e, \\
 V(x)&\to 0 &{\rm as} \ d(x)\to 0.
\end{aligned}\right.
\end{equation*}
\end{lemma}

In this section we use the fact that for any $\varphi\in
C(\mathcal S\times [0,T])$, there exists
\begin{equation}\label{e610}
\tilde \varphi \in C(\overline Q_T)\quad \textrm{such that}\;\,
\tilde \varphi=\varphi \quad \textrm{in}\;\; \mathcal S \times
[0,T]\,.
\end{equation}
We shall write $\tilde \varphi \equiv \varphi$.

\begin{proof}[\bf Proof of Theorem \ref{pointcond}]
The proof is divided into two main parts. At first,
we consider that case of a density $\rho$ satisfying hypothesis
{\bf H4} and
\begin{equation*}
\inf_{\Omega} \rho >0.
\end{equation*}

\noindent Let $\eta_0>0$. For any $0<\eta<\eta_0$, we define
$u^\eta_\e \in C(\overline{\Omega^\varepsilon} \times [0,T])$ as
the unique solution (see \cite{LSU}) to
\begin{equation}\label{e16bis}
\left\{
\begin{array}{ll}
\,   \rho\, \d_t u = \Delta\big[G(u)\big] &\textrm{in}\,\,\Omega^\varepsilon\times (0, T)\,, \\&\\
\, u\,=\, \varphi +\eta &\textrm{on}\,\, \mathcal
A^\varepsilon\times (0,T)\,,
\\ & \\ \, u \, = u_{0,\e} +\eta &\textrm{in}\,\, \Omega^\varepsilon\times \{0\}\,,
\end{array}
\right.
\end{equation}
where
\[u_{0,\e}(x):= \zeta_\e \, u_0(x) + (1-\zeta_\e) \, \varphi(x,0) \quad \textrm{in} \;\; \overline \Omega^\varepsilon\,,\]
and
$\{\zeta_\e\}\subset C^{\infty}_c(\Omega^\varepsilon)$ is a
sequence of functions such that, for any $\e>0$, $0\le
\zeta_\e\le 1$ and $ \zeta_\e\equiv 1$ in $\Omega^{2\e}$. By the comparison principle, there holds
\begin{equation}\label{e17bis}
|u^\eta_\e|\le K:=\max\{\|u_0\|_\infty, \|\varphi\|_\infty
\}+\eta_0\quad \textrm{in}\;\;\Omega^\varepsilon\times (0,T)\,.
\end{equation}

Moreover, by usual compactness arguments (see, e.g., \cite{LSU}),
there exists a subsequence
$\{u^\eta_{\e_k}\}\subseteq\{u^\eta_\e\}$ which converges, as
$\e_k\to 0$, locally uniformly in $\Omega\times [0,T]$, to a
solution $u^\eta$ to the following problem
\begin{equation}\label{e18bis}
\left\{
\begin{array}{ll}
\,   \rho\, \d_t u = \Delta \big[G(u)\big] &\textrm{in}\,\,
\Omega\times (0,T]\,,
\\& \\
\textrm{ }u \, = u_0+\eta& \textrm{in\ \ } \Omega\times \{0\} \,.
\end{array}
\right.
\end{equation}
We want to prove that, for each $\tau \in (0,T)$,
$$\lim_{\newatop{x\to x_0}{t\to t_0}} u^\eta(x,t)=\varphi(x_0,t_0),$$
 uniformly with respect to $t_0 \in (\tau,T]$ , $x_0 \in \mathcal S$ and $\eta \in (0,\eta_0)$.

\vskip0.2cm

Take any $\tau\in(0, T/2).$ Let $(x_0,t_0) \in \mathcal S \times
[2\tau,T]$. Set $N^\e_\delta(x_0):=B_{\delta}(x_0) \cap \Omega^\e$
for any $\delta
>0$ and $\e>0$ small enough. From the continuity of the function
$\varphi$ and since $G \in C^1(\R)$ is increasing, there follows
that,  for any $\sigma>0$, there exists $\delta(\sigma)>0$,
independent of $(x_0, t_0),$ such that
\begin{equation}\label{stima1}
G^{-1}\big[G(\varphi(x_0,t_0)+\eta)-\sigma\big]\, \le\, \varphi(x,t)+\eta \,\le
G^{-1}\big[G(\varphi(x_0,t_0)+\eta)+\sigma\big]\,,
\end{equation}
for all $(x,t)  \in \overline N_\delta(x_0) \times (\underline
t_\de,\overline t_\de)$, where
$$\underline t_\de:= t_0-\de\,, \quad {\rm and} \quad \overline t_\de:=
\min\{ t_0+\de, T\},$$ and
\[N_\delta(x_0):= B_\delta(x_0)\cap \Omega\,.\]
Clearly, $\underline t_\delta>\tau\,.$ Now, for any $(x,t)  \in
\overline N_\delta(x_0) \times (\underline t_\de,\overline
t_\de)$, we define
\begin{equation}\label{e600}
 \underline w(x,t) :=G^{-1}\big[-\underline{M} V(x) -\s
+G(\varphi(x_0,t_0)+\eta) -\underline{\l} (t-t_0)^2-\underline \beta |x-x_0|^2\big], \,
\end{equation}
with $V(x)$ as in Lemma \ref{lemma0} and $\underline M$, $\underline \lambda$ and $\underline \beta$ positive constants to be fixed conveniently in the sequel.

First of all we want to prove that
\begin{equation}\label{e305}
\rho \d_t \underline w \leq \Delta G(\underline w)\quad
\textrm{in}\;\; N_\de^\e(x_0) \times (\underline t_\de,\overline
t_\de).
\end{equation} To his purpose, we note that
\begin{equation*}
\rho \d_t \underline w \leq \rho \frac{2\underline \lambda \de}{\alpha_0}, \qquad {\rm and } \qquad
\Delta G(\underline w) \geq \underline M \rho -2 \underline \beta N.
\end{equation*}
Hence, the function $\underline w$ solves \eqref{e305}, if
\begin{equation}\label{condM}
\underline M \geq \frac{2 \underline \beta N}{\inf_{\Omega} \rho} + \frac{2\underline \lambda \de}{\alpha_0}.
\end{equation}
Going further, for any $(x,t) \in  [B_{\delta}(x_0) \cap \mathcal A^\e] \times (\underline t_\de,\overline t_\de)$, we have
\begin{equation}\label{stima3}
\underline w(x,t) \leq G^{-1}[G(\varphi(x_0,t_0)+\eta)-\sigma].
\end{equation}
Moreover, for $(x,t) \in  [\d B_{\delta}(x_0) \cap \Omega^\e] \times (\underline t_\de,\overline t_\de)$, there holds
\begin{equation}\label{stima4}
\underline w(x,t) \leq -K,
\end{equation}
provided
\begin{equation*}
\underline\beta \geq
\frac{G(||\varphi||_{L^\infty}+\eta_0)-G(-K)}{\delta^2}.
\end{equation*}
Finally, for all $(x,t) \in N_\de^\e(x_0) \times \{ \underline
t_\de\}$, there holds
\begin{equation}\label{stima5}
\underline w(x,t) \leq G^{-1}[G(\varphi(x_0,t_0)+\eta)-\underline \lambda \de^2] \leq -K,
\end{equation}
assuming
\begin{equation*}
\underline \lambda \geq \frac{G(||\varphi||_{L^\infty}+\eta_0)-G(-K)}{\de^2}.
\end{equation*}
From \eqref{stima3}, \eqref{stima4} and \eqref{stima5} we obtain
that $\underline w$ is a subsolution to the following problem
\begin{equation}\label{subsoluz2bis}
\left\{
\begin{array}{ll}
\,   \rho\, \d_t u = \Delta\big[G(u)\big] &\textrm{in}\,\,\ N^{\varepsilon}_{\delta}(x_0)\times (\underline t_\de, \overline t_\de)\,, \\&\\
\, u\,=\, -K &\textrm{in}\,\, [\d B_{\delta}(x_0) \cap
\Omega^\e]\times (\underline t_\de, \overline t_\de)\,,
\\ & \\ \, u \, = G^{-1}[G(\varphi+\eta)-\sigma] &\textrm{in}\,\, [ B_{\delta}(x_0) \cap\mathcal A^\e]\times(\underline t_\de, \overline t_\de)\,,
\\ & \\ \, u \, = -K&\textrm{in}\,\, N^{\varepsilon}_\de(x_0)\times \{\underline t_\de\}\,.
\end{array}
\right.
\end{equation}
\vskip0.3cm \noindent Recalling the definition of $u^\eta_\e$
given in \eqref{e16bis}, and by using \eqref{e17bis}, it follows
that $u^\eta$ is a supersolution to problem \eqref{subsoluz2bis}.
Note that sub-- and supersolutions to problem \eqref{subsoluz2bis}
are meant similarly to Definition \ref{soluzdeboleeps},
considering that $N^\e_\delta(x_0)$ is piece-wise smooth; the same
holds for problems of the same form we mention in the sequel.

\smallskip

By proceeding with the same methods, for all $(x,t)  \in \overline
N_\delta(x_0) \times (\underline t_\de,\overline t_\de)$ we define
\begin{equation}\label{e601}
 \overline w(x,t) :=G^{-1}\big[\overline{M} V(x) +\s
+G(\varphi(x_0,t_0)+\eta) +\overline{\l} (t-t_0)^2+\overline \beta |x-x_0|^2\big], \,
\end{equation}
proving that, with an appropriate choice for the coefficients
$\overline M, \overline \lambda$ and $\overline \beta$, $\overline
w$ is a supersolution to problem
\begin{equation}\label{subsoluz3bis}
\left\{
\begin{array}{ll}
\,   \rho\, \d_t u = \Delta\big[G(u)\big] &\textrm{in}\,\,\ N^{\varepsilon}_{\delta}(x_0)\times (\underline t_\de, \overline t_\de)\,, \\&\\
\, u\,=\, K &\textrm{in}\,\, [\d B_{\delta}(x_0) \cap
\Omega^\e]\times (\underline t_\de, \overline t_\de)\,,
\\ & \\ \, u \, = G^{-1}[G(\varphi+\eta)+\sigma] &\textrm{in}\,\, [ B_{\delta}(x_0) \cap\d \Omega^\e]\times(\underline t_\de, \overline t_\de)\,,
\\ & \\ \, u \, = K&\textrm{in}\,\, N^{\varepsilon}_\de(x_0)\times \{\underline t_\de\}\,.
\end{array}
\right.
\end{equation}
Precisely, we require $\overline M$ to be such that
\begin{equation*}
\overline M \geq \frac{2 \overline \beta N}{\inf_{\Omega} \rho} + \frac{2\overline \lambda \de}{\alpha_0},
\end{equation*}
while $\overline \beta$ and $\overline \lambda$ are chosen so that
\begin{equation*}
\overline\beta \geq \frac{G(K)-G(||\varphi||_{L^\infty}+\eta_0)}{\delta^2}, \qquad \underline \lambda \geq \frac{G(K)-G(||\varphi||_{L^\infty}+\eta_0)}{\de^2}.
\end{equation*}
On the other hand, $u^\eta$ is a subsolution to problem
\eqref{subsoluz3bis}. Hence, by the comparison principle, and by
letting $\e_k\to 0$, we get
\begin{equation}\label{e36}
\underline w \le u^\eta \le \overline w\quad \textrm{in}\;\;
 N_\de(x_0) \times (\underline t_\delta,\overline t_\de)\,.
\end{equation}
Take any $\tau\in (0,T/2)$ and $(x_0,t_0) \in\mathcal S \times
[2\tau,T]$. Due to \eqref{e36}, recalling the definition of
$\underline w$ and $\overline w$ and by letting $x\to x_0, t\to
t_0$, one has
\begin{equation*}
G^{-1}\big[G(\varphi(x_0,t_0)+\eta)-2\s \big]\le u^\eta(x_0,t_0)
\le G^{-1}\big[G(\varphi(x_0,t_0)+\eta)+2\s\big] \,.
\end{equation*}
Letting $\s \to 0^+$, we end up with
\begin{equation*}
\lim_{\newatop{x\to x_0}{t\to t_0}} u^\eta(x,t)=\varphi(x_0,t_0),
\end{equation*}
uniformly with respect to $t_0\in (2\tau,T)$, $x_0 \in \mathcal S$
and $\eta \in (0,\eta_0)$, for each $\tau \in (0,T/2)$. Moreover,
by usual compactness arguments, there exists a subsequence $\{
u^{\eta_k} \} \subset \{ u^\eta \} $ which converges, as $\eta_k
\to 0$, to a solution $u$ to \eqref{mainproblem}, locally
uniformly in $\Omega \times [0,T]$. Hence, by using \eqref{e36},
we have, in the limit $\sigma \to 0^+$ and $\eta \to 0^+$,
\begin{equation*}
\lim_{\newatop{x\to x_0}{t\to t_0}}u(x,t)=\varphi(x_0,t_0),
\end{equation*}
uniformly with respect to $t_0\in (2\tau,T)$ and $x_0 \in \mathcal
S$, for each $\tau \in (0,T/2)$. \vskip0.2cm It remains to show that $u$ is the maximal solution. To
this end, let $v$ be any solution to problem \eqref{mainproblem}
satisfying \eqref{primoris}. From \eqref{e36} it follows that for
any $\a \in (0, \eta_0/4)$ and for any $\tau \in (0,T)$, there
exists $\tilde \e>0$ such that for any $0<\e<\tilde \e$ and
$\eta\in (0,\eta_0)$
\begin{equation}\label{e30c}
v(x,t)\leq \varphi(x,t)+\a \leq u^\eta(x,t)\quad \textrm{for
all}\;\; (x,t)\in \mathcal A^\e\times(\tau,T]\,.
\end{equation}
Moreover
\begin{equation}\label{e31c}
v(x,0)= u_0(x)<u_0(x)+\eta=u^\eta(x,0)\quad \textrm{for all}\;\;
x\in \Omega.
\end{equation}
Since $v(x,t)$ and $u^\eta(x,t)$ are solutions to the same
equations in $\Omega\times (0,T]$, in view of \eqref{e30c},
\eqref{e31c} and Lemma \ref{lemma-1} there holds
\[ v(x,t)\leq u^\eta(x,t)\quad \textrm{for all}\;\; (x,t)\in Q_T\,.\]
Passing to the limit $\eta \to 0^+$ we obtain
$$v\leq u\quad \textrm{in}\;\;Q_T,$$
and the proof is complete, in this case.

\vskip0.5cm

In the second part of the proof, we consider a density $\rho$ such
that $\rho \in L^{\infty}(\Omega)$. Now, we need to slightly
modify the arguments used above.
 Since $\mathcal S \in C^1$, by \cite{GT} the uniform exterior sphere condition is satisfied, i.e. there exists $R>0$ such that for any $x_0 \in \mathcal S$ we can find
 $x_1 \in \R^N\setminus \bar\Omega$ such that
$B(x_1, R)\subset \R^N\setminus \bar\Omega$ and
$\overline{B(x_1,R)} \cap \mathcal S = \{ x_0\}$. Thus, by
standard arguments (see K. Miller \cite{Mi}), it is proven that
the following function
\begin{equation}\label{defhmiller}
h(x):= C[e^{-a \, R^2}-e^{-a \, |x-x_0|^2}]
\end{equation}
satisfies
\begin{itemize}
\item $\Delta h \leq -1$  in $B_R(x_0)$; \vskip0.2cm \item $h>0$
for all  $x \in \big[\bar B_{R}(x_0) \cap \bar
\Omega\big]\setminus\{x_0\}$; \vskip0.2cm \item $h(x_0)=0$,
\end{itemize}
for a suitable choice of the constants $C>0$ and $a>0$,
independent of $x_0 \in \mathcal S$.

\vskip0.3cm

The function $h(x)$  can be used  in order to built suitable
barrier functions  $\underline w(x,t)$ and $\overline w(x,t)$. To
this end, for $(x,t)\in \overline N_\delta(x_0) \times (\underline
t_\de,\overline t_\de)$, we define
\begin{equation}\label{e602}
 \underline w(x,t) :=G^{-1}\big[-\underline{M} h(x) -\s
+G(\varphi(x_0,t_0)+\eta) -\underline{\l} (t-t_0)^2\, \big],
\end{equation}
being $h(x)$ as in \eqref{defhmiller}.

\vskip0.2cm First of all, because of the properties of $h(x)$,
there holds $\rho \d_t \underline w \leq \Delta G(\underline w)$,
if
\begin{equation*}
\underline M \geq \frac{2 \,\rho(x) \,\underline\lambda
\,\delta}{\alpha_0},
\end{equation*}
Hence, we require that
\begin{equation*}
\underline M \geq \frac{2 \,  \lambda \, \delta}{\alpha_0} \, \| \rho \|_{L^\infty}.
\end{equation*}
Next, let $(x,t) \in  [B_{\delta}(x_0) \cap \mathcal A^\e] \times (\underline t_\de,\overline t_\de)$; we have
\begin{equation}\label{stima3bis}
\underline w \leq G^{-1}[G(\varphi(x_0,t_0)+\eta)-\sigma].
\end{equation}
Moreover, for $(x,t) \in  [\d B_{\delta}(x_0) \cap \Omega^\e] \times (\underline t_\de,\overline t_\de)$ we have
\begin{equation}\label{stima4bis}
\underline w(x,t) \leq -K,
\end{equation}
provided
\begin{equation*}
\underline M \geq
\frac{G(||\varphi||_{L^\infty}+\eta_0)-G(-K)}{\inf_{\d
B_{\delta}(x_0) \cap \Omega} h} \, .
\end{equation*}
Finally, for $(x,t) \in N_\de^\e(x_0) \times \{\underline t_\de\}$
\begin{equation}\label{stima5bis}
\underline w(x,t) \leq G^{-1}[G(\varphi(x_0,t_0)+\eta)-\underline \lambda \de^2] \leq -K
\end{equation}
imposing
\begin{equation*}
\underline \lambda \geq \frac{G(||\varphi||_{L^\infty}+\eta_0)-G(-K)}{\de^2}.
\end{equation*}

\vskip0.3cm
From \eqref{stima3bis}, \eqref{stima4bis} and \eqref{stima5bis} we can state  that $\underline w$ is a subsolution to the following problem
\begin{equation}\label{subsoluz2bisbis}
\left\{
\begin{array}{ll}
\,   \rho\, \d_t u = \Delta\big[G(u)\big] &\textrm{in}\,\,\ N^{\varepsilon,\e_0}\times (\underline t_\de, \overline t_\de)\,, \\&\\
\, u\,=\, -K &\textrm{on}\,\, [\d B_{\delta}(x_0) \cap \Omega^\e]\times (\underline t_\de, \overline t_\de)\,,
\\ & \\ \, u \, = G^{-1}[G(\varphi+\eta)-\sigma] &\textrm{in}\,\, [ B_{\delta}(x_0) \cap\d \Omega^\e]\times(\underline t_\de, \overline t_\de)\,,
\\ & \\ \, u \, = -K&\textrm{in}\,\, N^{\varepsilon}_\de(x_0)\times \{\underline t_\de\}\,,
\end{array}
\right.
\end{equation}
while $u^\eta$ is a supersolution to the same problem. By
proceeding with the same methods, for all $(x,t) \in \overline
N_\delta(x_0) \times (\underline t_\de,\overline t_\de)$ we define
\begin{equation}\label{e603}
 \overline w(x,t) :=G^{-1}\big[\overline{M} \, h(x) +\s
+G(\varphi(x_0,t_0)+\eta) +\overline{\l} (t-t_0)^2\big], \,
\end{equation}
proving that, with the appropriate choices for the coefficients
$\overline M, \overline \lambda$ and $\overline \beta$, $\overline
w$ is a super-solution to  problem
\begin{equation}\label{subsoluz2bisbisa}
\left\{
\begin{array}{ll}
\,   \rho\, \d_t u = \Delta\big[G(u)\big] &\textrm{in}\,\,\ N^{\varepsilon,\e_0}\times (\underline t_\de, \overline t_\de)\,, \\&\\
\, u\,=\, K &\textrm{on}\,\, [\d B_{\delta}(x_0) \cap
\Omega^\e]\times (\underline t_\de, \overline t_\de)\,,
\\ & \\ \, u \, = G^{-1}[G(\varphi+\eta)+\sigma] &\textrm{in}\,\, [ B_{\delta}(x_0) \cap\d \Omega^\e]\times(\underline t_\de, \overline t_\de)\,,
\\ & \\ \, u \, = K&\textrm{in}\,\, N^{\varepsilon}_\de(x_0)\times \{\underline t_\de\}\,,
\end{array}
\right.
\end{equation}
while $u^\eta$ is a subsolution to the same problem. Hence, by the
comparison principle, and by letting $\e_k\to 0$, we get
\begin{equation}\label{e36bis}
\underline w \le u^\eta \le \overline w\quad \textrm{in}\;\;
 N_\de(x_0) \times (\underline t_\delta,\overline t_\de)\,.
\end{equation}
Take any $\tau\in (0, T/2)\,.$ Let $(x_0,t_0) \in \mathcal S
\times [2\tau,T]$. In view of \eqref{e36bis}, recalling the
definition of $\underline w$ and $\overline w$ and by letting
$x\to x_0$ and choosing $t=t_0$, one has
\begin{equation*}
G^{-1}\big[G(\varphi(x_0,t_0)+\eta)-2\s \big]\le u^\eta(x,t_0) \le
G^{-1}\big[G(\varphi(x_0,t_0)+\eta)+2\s\big] \,.
\end{equation*}
So, the thesis follows for $\s \to 0^+$ as in the previous case,
as well as the maximality of $u$.

\end{proof}

\begin{proof}[\bf Proof of Theorem \ref{degpuntuale}]

As in the proof of Theorem \ref{pointcond}, we consider at first
the case of a density $\rho$ satisfying hypothesis {\bf H4} and
$\inf_{\Omega} \rho >0$.

\smallskip
We define $u^\eta_\e \in C(\overline{\Omega^\varepsilon} \times
[0,T])$ as the unique solution to \eqref{e16bis}. Take any $x_0\in
\mathcal S$. Observe that from \eqref{e301} we can infer that for
any $\sigma>0$ there exists $\delta=\delta(\sigma)>0$, independent
of $x_0$, such that
\begin{equation}\label{e306}
G^{-1}\big[G(\varphi(x_0)+\eta) -\sigma\big]\leq u_0(x)+\eta\leq
G^{-1}\big[G(\varphi(x_0)+\eta) +\sigma\big]\quad \textrm{for
all}\;\, x\in N_\delta(x_0)\,.
\end{equation}

\smallskip

\noindent For all $x \in \overline N_\delta(x_0)$, we  define
\begin{equation}\label{ea3}
\underline w (x):=G^{-1}\big[-\underline{M} \, V(x) - \sigma +
G(\varphi(x_0)+\eta) -\underline \beta|x-x_0|^2\big]\,,
\end{equation}
where $V$ is defined in Lemma \ref{lemma0}, and $\underline M$ and $\underline \beta$ are positive constants to be chosen. There holds
\begin{equation*}
\Delta G(\underline w) \geq \underline M \rho-2 \underline \beta N \geq0,
\end{equation*}
provided
\begin{equation}\label{M2}
\underline M \geq \frac{2 \underline \beta N}{\inf_{\Omega} \rho}.
\end{equation}
Going further, for all $(x,t)  \in [B_{\delta}(x_0) \cap \mathcal
A^\e] \times (0,T)$, there holds
\begin{equation}\label{e59}
\underline w \leq \varphi(x_0)+\eta,
\end{equation}
while, for all $(x,t) \in  [\d B_{\delta}(x_0) \cap \Omega^\e]
\times(0,T)$
\begin{equation*}
\underline w \leq -K,
\end{equation*}
provided
\begin{equation*}
\underline \beta \geq \frac{G(|\varphi(x_0)|)-G(-K)}{\delta^2}.
\end{equation*}
Moreover, from \eqref{e306} it follows that
\begin{equation}\label{e307}
\underline w(x)\leq u_0(x)+\eta \quad \textrm{for all}\;\; x\in
N^\e_\delta(x_0)\,.
\end{equation}
Thus $\underline w$ is a subsolution, while $u^\eta$ is a
supersolution to problem
\begin{equation}\label{e308}
\left\{
\begin{array}{ll}
\,   \rho\, \d_t u = \Delta\big[G(u)\big] &\textrm{in}\,\,\ N^{\varepsilon}_\delta(x_0)\times (0,T)\,, \\&\\
\, u\,=\, -K &\textrm{on}\,\, [\d B_{\delta}(x_0) \cap
\Omega^\e]\times (0,T)\,,
\\ & \\ \, u \, = G^{-1}[G(\varphi+\eta)-\sigma] &\textrm{in}\,\, [ B_{\delta}(x_0) \cap\d \Omega^\e]\times(0, T)\,,
\\ & \\ \, u \, = u_0+\eta   &\textrm{in}\,\, N^{\varepsilon}_\de(x_0)\times \{0\}\,,
\end{array}
\right.
\end{equation}
By the comparison principle, there holds
\begin{equation}\label{e64bis}
\underline w \leq u^\eta_\e\quad \textrm{in}\;\;
N^\e_{\delta}(x_0)\times (0, T)\,.
\end{equation}
Analogously, we have
\begin{equation}\label{e66bis}
u^\eta_\e \leq \overline{w} \quad \textrm{in}\;\;
N^\e_{\delta}(x_0)\times (0, T)\,,
\end{equation}
where
\begin{equation}\label{e604}
\overline w(x):=G^{-1}\big[\overline{M} V(x) + \s+
G(\varphi(x_0)+\eta) \big]\,,
\end{equation}
with $\overline M>0$ conveniently chosen.

\smallskip

From \eqref{e64bis} and \eqref{e66bis} with $\e=\e_k \to 0$, we
obtain
\begin{equation*}
\underline w \le u^\eta \le \overline w\quad \textrm{in}\;\;
 N_{\delta}(x_0)\times
(0, T)\,,
\end{equation*}
where $u^\eta$ is a solution to problem \eqref{e18bis}. Hence the
thesis follows by letting $x\to x_0$ and $\sigma \to 0^+$, as in
the proof of Theorem \ref{pointcond}.

By slightly modifying the previous arguments, it is possible to
prove Theorem \ref{degpuntuale} also in the case of a density
$\rho$ satisfying $\rho \in L^\infty(\Omega)$. Indeed, we
construct the barrier functions $\underline w(x)$ and $\overline
w(x)$ as
\begin{equation}\label{e606}
\underline w(x):=G^{-1}\big[-\underline{M} h(x) -\s
+G(\varphi(x_0)+\eta) -\underline{\beta} |x-x_0|^2\, \big],
\end{equation}
\begin{equation}\label{e607}
\overline w(x):=G^{-1}\big[\overline{M} \, h(x) +\s
+G(\varphi(x_0)+\eta) +\overline{\beta} |x-x_0|^2\big],
\end{equation}
being $h(x)$ as in \eqref{defhmiller}. The thesis follows as in
the second part of the proof of Theorem \ref{pointcond}, by making
use of the properties of $h(x)$ and by suitable choices of the
constants $\underline M, \underline \b, \overline M, \overline
\b$.
\end{proof}

\begin{proof}[\bf Proof of Theorem \ref{deg2puntuale}]

Let $$\alpha_2:=\min\Big\{\min_{\bar \Omega \times [0,T]}
\varphi,\,  \alpha_1\Big\}\,,$$ with $\alpha_1>0$ as in
\eqref{ipophipositiva}. Since $\varphi\in C(\mathcal
S\times[0,T])$ and $\varphi>0$ in $\mathcal S\times [0,T]$, we can
select $\tilde\varphi\equiv \varphi$ as in \eqref{e610}, such that
$\tilde\varphi>0$ in $\bar Q_T$\,. So, $\alpha_2>0.$ Take
$\underline u_0\in C(\bar \Omega)$ such that
\begin{equation}\label{eq1}
\underline u_0\leq u_0\quad \textrm{in}\;\; \Omega\,,\;\,
\lim_{x\to x_0} \underline u_0(x)=\frac{\alpha_2}{2}\,.
\end{equation}
By Theorem \ref{degpuntuale}, there exists a solution $\underline
u(x,t)$ to the following problem
\begin{equation}\label{eq2bis}
\left\{\begin{aligned}
\rho \d_t u &= \Delta [G(u)] \quad & {\rm in } \ & \Omega \times (0,T], \\
u&=\underline u_0 \quad &{\rm in}  \ & \Omega \times \{ 0\},
\end{aligned}\right.
\end{equation}
such that
\begin{equation}\label{eq3bis}
\lim_{x\to x_0} \underline u(x,t)=\frac{\alpha_2}2 \quad
\textrm{uniformly for}\;\;x_0\in\mathcal S, t\in [0,T]\,.
\end{equation}
We construct the approximating sequence $\{u^\eta_\e\}$ as in the
proof of Theorem \ref{pointcond}. Due to \eqref{eq2bis} and
\eqref{eq3bis}, by the comparison principle, we have that for some
$\e_0>0$, for every $0<\e<\e_0$
\begin{equation}\label{eq4bis}
\underline u(x,t)\leq u^\eta_\e(x,t) \quad \textrm{for all}\; x\in
\Omega^\e\,,\ t\in (0,T]\,.
\end{equation}
Then there exists a subsequence $\{u^\eta_{\e_k}\}\subset
\{u^\eta_\e\}$ which converges, as $\e_k\to0$, to a solution
$u^\eta$ to \eqref{e18bis}. From \eqref{eq4bis} it follows that
\[
u^\eta(x,t)\geq \underline u(x,t)\quad \textrm{for all}\;\; x\in
\Omega\,, \ t\in (0,T]\,.
\]
Therefore, for some $0<\e_1<\e_0$, for all $0<\eta<\eta_0$ there
holds
\begin{equation}\label{eq5bis}
u^\eta(x,t)\geq \frac{\alpha_2}4\quad \textrm{for all}\; x\in
\mathcal S^{\e_1}, \ t\in (0,T]\,.
\end{equation}
Hence, in $\mathcal S^{\e_1}\times (0,T]$ the equation does not
degenerate, i.e., for some $\a_0>0$,
\[G'(u)\geq \a_0 \quad \textrm{in}\;\;\mathcal S^{\e_1}\times (0,T]\,. \]

Select a function $G_1$ such that hypothesis {\bf H2} is
satisfied; moreover, $G_1(u)=G(u)$ for $u\geq \frac{\a_2}4$ and
$G'_1(u)\geq \frac{\a_0}{2}>0$ for all $u\in \R.$ From
\eqref{eq5bis}, $u^\eta(x,t)$ is a solution to the non-degenerate
equation
\[\rho \d_t u =\big[G_1(u)\big]\quad \textrm{in}\;\;\mathcal S^{\e_1}\times (0,T]\,. \]
Thus we get the conclusion as in the proof of Theorem
\ref{pointcond}\,.

\end{proof}

\section{uniqueness results: proofs}

The proof of  Theorem \ref{teounicita} makes use of the following lemma.

\begin{lemma}\label{lem3}
Let $\e_0>0$ and $ F\in C^{\infty}(\Omega)$ such that $F\geq 0,\;
{\rm supp} \ F\subset \Omega^{\e_0}$. Then, for any $0<\e<\e_0$,
there exists a unique classical solution $\psi^\e$ to the problem
\begin{equation}\label{e60}
\left\{
\begin{array}{ll}
\,   \Delta \psi^\e = - F
&\textrm{in}\,\, \Omega^\e
\\& \\
\textrm{ }\psi^\e \, = 0& \textrm{on\, } \mathcal A^\e\,.
\end{array}
\right.
\end{equation}
Moreover, for any $0<\e<\e_0$ there holds:
\begin{equation}\label{e61}
\psi^\e >  0\quad \textrm{in}\;\; \Omega^\e\,;
\end{equation}
\begin{equation}\label{e62}
\langle\nabla \psi^\e(x), \nu^\e(x)\rangle < 0 \quad \textrm{for all}\;\; x\in \mathcal A^\e\,;
\end{equation}
\begin{equation}\label{e63}
\int_{\mathcal A^\e} \big| \langle \nabla \psi^\e, \nu^\e\rangle \big|d S \leq \bar C\,,
\end{equation}
for some constant $\bar C>0$ independent of $\e$; here $\nu^\e$
denotes the outer unit normal vector to $\d \Omega^\e$.
\end{lemma}

\begin{proof}

For any $0<\e<\e_0$, the existence and the uniqueness of the
solution $\psi_\e$ to \eqref{e60} follow immediately. Moreover,
since $F\geq 0$, by the strong maximum principle we get
\eqref{e61} and \eqref{e62}. Observe that, since $ {\rm supp} \
F\subset \Omega^{\e_0}$, then for any $0<\e<\e_0$ we have
\begin{equation}\label{e64}
 \int_{\Omega^\e} F(x) \, dx = \int_{\Omega^{\e_0}} F(x) \, dx =: \bar C\,.
 \end{equation}
On the other hand, from \eqref{e60} by integrating by parts,
\begin{equation}\label{e65}
\int_{\Omega^\e} F(x) dx = - \int_{\Omega^\e}  \Delta \psi^\e dx =
-\int_{\mathcal A^\e} \langle \nabla \psi^\e, \nu^\e\rangle d S
\,.
\end{equation}
From \eqref{e64}, \eqref{e65}, and \eqref{e62} we get
\eqref{e63}\,.

\end{proof}

\smallskip
\bigskip

\begin{proof}[\bf Proof of Theorem \ref{teounicita}]

 In view of the
hypotheses we made, we can apply Theorem \ref{pointcond} to infer that there exists a maximal solution $\bar u$ to  \eqref{mainproblem}. Let $u$ be any solution to \eqref{mainproblem}, and let $F \in C^{\infty}_c(\Omega)$.

Without loss of
generality, we suppose ${\rm supp \ } F \subset  \Omega^{\e_0},$ for some $\e_0
>0$, $F \not \equiv 0$ and $ F\geq 0$. Since both $\bar u$ and  $u$ solves \eqref{mainproblem}, we apply the  equality
\eqref{e8a} with $\Omega=\Omega^\e$, $0<\e<2\e_0$ and $
\psi(x,t)=\psi^\e(x)$. We get

\begin{equation*}
\begin{aligned}
\int_0^T \int_{\Omega^\e} & [G(\bar u)- G(u)] \, F(x)\, dx \, dt  =\\
&=- \int_{\Omega^\e} [\bar u(x,T)-u(x,T)] \rho(x) \, \psi^\e(x) dx
-\int_0^T \int_{\mathcal A^\e} [G(\bar u)- G(u)] \langle \nabla
\psi^\e, \nu^\e \rangle dS \, dt
\end{aligned}
\end{equation*}
Since $F\geq 0,\; \psi^\e\geq 0$, $\bar u\geq u$ in $\Omega^\e$
and $\langle \nabla \psi, \nu^\e\rangle \leq 0$ on $\mathcal
A^\e$, the previous equality gives:
\begin{equation}\label{mandoepsazero}
\begin{aligned}
 \int_0^T& \int_{\Omega^\e}[G(\bar u) - G(u)] F(x)\, dx \, dt \leq  -\int_0^T\int_{\mathcal
A^\e}  [G(\bar u)-G(u)] \langle \nabla \psi, \nu^\e\rangle dS\, dt \, = \\
 & =  -\int_0^\tau\int_{\mathcal
A^\e}  [G(\bar u)-G(u)] \langle \nabla \psi, \nu^\e\rangle dS\, dt
\,   -\int_\tau^T\int_{\mathcal
A^\e}  [G(\bar u)-G(u)] \langle \nabla \psi, \nu^\e\rangle dS\, dt \, \\
\end{aligned}
\end{equation}
Going further, by \eqref{e63}, we get
\begin{equation}\label{e67}
\begin{aligned}
\int_\tau^T \int_{\Omega^\e}[G(\bar u) - G(u)] F(x) \, dx \, dt
&\leq \sup_{\mathcal A^\e\times (\tau,T)}[G(\bar u)- G(u)]
\int_{\mathcal A^\e}\big| \langle \nabla \psi, \nu^\e\rangle\big|
dS \, dt \\
&\leq \bar C  \sup_{\mathcal A^\e\times (\tau,T)}[G(\bar u)-
G(u)]\,.
 \end{aligned}
\end{equation}
Furthermore
\begin{equation}\label{e67prime}
\int_0^\tau \int_{\Omega^\e}[G(\bar u) - G(u)] F(x) \, dx \, dt
\leq \bar C \, \tau \, C,
\end{equation}
where the constant $C$ only depends on $\| u\|_{L^\infty}$ and $\| \bar u\|_{L^\infty}$.
Since any solution to \eqref{mainproblem} satisfies condition \eqref{primoris} uniformly for $t \in [\tau,T]$, for each $\tau \in (0,T)$, we get
\begin{equation}\label{e68}
\sup_{\mathcal A^\e\times (\tau,T)}[G(\bar u)- G(u)]\to 0 \quad
\textrm{as}\;\; \e\to 0\,.
\end{equation}
Hence, in view of \eqref{e67}, \eqref{e67prime} and \eqref{e68}, if we let $\e \to 0$ in \eqref{mandoepsazero} and then $\tau \to 0$, we obtain
\begin{equation}\label{e66}
\int_0^T\int_{\Omega} \big[ G(\bar u) - G(u) \big]\, F(x) \, dx \, dt\,= 0\,.
\end{equation}
In view of the hypothesis {\bf H2}, and because of the arbitrariness of $F$, \eqref{e66} implies
\begin{equation*}
 \bar u \,= u \quad \textrm{in}\;\; \Omega \times (0,T]\,,
\end{equation*}
and the proof is completed.

\end{proof}

As outlined in Remark \ref{osst1}, Theorem \ref{teounicita} holds
true either if we consider a non degenerate nonlinearity $G$
satisfying hypothesis {\bf H5} or if we suppose
\begin{equation*}
\varphi(x_0,t) \equiv \varphi(x_0), \quad {\rm for \ all } \ t \in[0,T].
\end{equation*}
Infact, in both cases, Theorem \ref{pointcond} and Theorem
\ref{deg2puntuale} assure the existence of the maximal solution
satisfying
 \eqref{primoris} and \eqref{secris} respectively. Hence, the uniqueness follows as in the proof of Theorem
 \ref{teounicita}.

\end{document}